\documentclass[12pt,reqno]{amsart}
\usepackage{amsfonts,amsmath,amssymb,graphicx,array,float}
\usepackage{verbatim}
\usepackage[english]{babel}
\usepackage{amsthm}
\usepackage{amsmath}
\usepackage{amssymb}
\usepackage{physics}
\usepackage[T1]{fontenc}
\usepackage{mathrsfs}
\usepackage{tensor}
\usepackage{url}
\usepackage[utf8]{inputenc}
\usepackage[a4paper,margin=0.75in]{geometry}

\newtheorem{thm}{Theorem}[section]
\newtheorem{lemma}{Lemma}[section]

\theoremstyle{remark}
\newtheorem{remark}{Remark}

\allowdisplaybreaks
\usepackage{comment}
\usepackage{theoremref}
\usepackage{hyperref}
\usepackage[msc-links, alphabetic]{amsrefs}

\theoremstyle{plain} \newtheorem{Thm}{Theorem}[section]
\theoremstyle{plain} 
\theoremstyle{plain} \newtheorem{Prop}[Thm]{Proposition}
\theoremstyle{plain} 
\theoremstyle{definition} 
\theoremstyle{definition} 
\theoremstyle{definition} 
\theoremstyle{definition} 

\numberwithin{equation}{section}
\numberwithin{remark}{section}

\begin{document}

\makeatletter
\title[Injectivity of spherical means]{Injectivity of spherical means on $H$-type groups}

\author{E. K. Narayanan}
\address{Department of Mathematics, Indian Institute of Science, Bangalore-12, India.}
\email{{naru@iisc.ac.in}}

\author{P. K. Sanjay}
\address{Department of Mathematics, National Institute of Technology, Calicut, India}
\email{sanjaypk@nitc.ac.in}

\author{K. T. Yasser}
\address{Department of Mathematics, National Institute of Technology, Calicut, India}
\email{yasser\_p160082ma@nitc.ac.in}

\thanks{The first named author thanks SERB, India for the financial support through MATRICS grant MTR/2018/00051}
\thanks{The second named author thanks SERB, India for the financial support through MATRICS grant MTR/2017/000741}
\thanks{The third named author thanks University Grants Commission (UGC) of India for the financial support}

\date{}

\begin{abstract}
We establish injectivity results for three different spherical means on an $H$-type group, $G$. First is the standard spherical means which is defined to be the average of a function over the spheres in the complement of the center, second is the average over the product of spheres in the center and its complement, and the third is the average over the spheres defined by a homogeneous norm on $G$. If $m$ is the dimension of the center of $G$, injectivity of these spherical means is proved for the range $1 \leq p \leq \frac{2m}{m-1}$. Examples are provided to show the sharpness of our results in the first two cases.
\end{abstract}

\maketitle

\tableofcontents

\section{Introduction}

One of the problems in Integral Geometry is to determine whether a function can be determined from its averages on spheres of a fixed radius $r>0$. This leads to the question of injectivity of the so called spherical mean operator. Let $\nu_r$ be the normalized surface measure on the sphere $\{x\in \mathbb{R}^n:|x|=r \}$ in $ \mathbb{R}^n$. The spherical means of a function $f$ is then defined to be the convolution $f\ast\nu_r$: \[f\ast\nu_r(x)=\int\limits_{|x|=r}f(x-y)d\nu_r(y).\]
The above is nothing but the average of the function $f$ over the sphere of radius $r$ centered at the point $x.$
The injectivity question is then the following:

Suppose that $ f \ast \nu_r(x) = 0$ for all $x \in \mathbb R^n,$ does it follow that $f$ is identically zero?
%\[ f*\nu_r(x)=0,\, \forall x\implies f\equiv 0?\]

In general, the answer to this question is no. 
For $\lambda>0$, let \[\varphi_\lambda(x)= c~\frac{J_{\frac{n}{2}-1}(\lambda|x|)}{(\lambda|x|)^{\frac{n}{2}-1}},~x \in \mathbb R^n,\]where $J_\alpha$ denotes the Bessel function of order $\alpha$ and $c$ is a constant that makes $\varphi_\lambda(0) =1$. Then it is well known that \[ \varphi_\lambda\ast\nu_r(x)=\varphi_\lambda(r)\varphi_\lambda(x),~x \in \mathbb R^n.\]
Hence, if $r> 0$ is a zero of the function $s \to J_{\frac{n}{2}-1}(\lambda s)$ (which exists) then $\varphi_\lambda\ast\nu_r$ is identically zero. On the other hand, if we consider averages over spheres of two different radii $r, s > 0,$ then a two radius theorem is true, provided $r/s$ is not a quotient of the zeroes of the Bessel function $J_{\frac{n}{2}-1}(t).$ That is, if both the convolutions $f \ast \nu_r$ and $f \ast \nu_s$ vanish identically, $f$ too vanishes identically provided $r/s$ is not a quotient of the zeroes of the Bessel function $J_{\frac{n}{2}-1}(t)$ (see \cite{MR562919}).

It is known that $\varphi_\lambda\in L^p(\mathbb{R}^n)$ if and only if $p>\frac{2n}{n-1}$. It follows that injectivity fails for $\frac{2n}{n-1}<p\leq \infty$. In \cite{MR1269222}, a one radius theorem is proved, which establishes the injectivity for the range $1\leq p\leq \frac{2n}{n-1}$. In other words, if $f \in L^p(\mathbb R^n)$ and $f \ast \nu_r$ is identically zero for a fixed radius $r > 0,$ then $f$ vanishes identically, provided $1 \leq p \leq \frac{2n}{n-1}.$

Now, consider the Heisenberg group $\mathbb{H}^n=\mathbb{C}^n\times \mathbb{R}$ with the group law \[(z,t)(w,s)=(z+w,t+s+\frac{1}{2}\Im(z\cdot\overline{w})), \]
 which makes $\mathbb{H}^n,$ a step two nilpotent Lie group. Let $\mu_r$ denote the normalized surface measure on the sphere $\{z\in\mathbb{C}^n: |z|=r \} $ viewed as a distribution on $\mathbb{H}^n$. The spherical means of a function $f$ on $\mathbb{H}^n$ is then defined to be $f\ast\mu_r(z,t)$: \[f\ast\mu_r(z,t)=\int\limits_{|w|=r}f(z-w,t-\frac{1}{2}\Im(z\cdot\overline{w}))\, d\mu_r(w). \]
 
 Using  the spectral decomposition of the sublaplacian on the Heisenberg group and the results proved by Strichartz in \cite{MR1101262}, Thangavelu \cite{MR1269222} proved the injectivity of the spherical mean operator for $L^p(\mathbb{H}^n),\, 1\leq p<\infty$. That is, if $f\in L^p(\mathbb{H}^n),1\leq p<\infty$ and $f\ast\mu_r(z,t)=0$ for all $(z,t)\in \mathbb{H}^n$, then $f$ vanishes identically. For $p = \infty,$ one has a two radius theorem for $\mathbb H^n$ which is proved using a Wiener-Tauberian theorem for the radial functions on the Heisenberg group (see \cite{MR1269218}). We refer the reader to \cite{MR1241128} and \cite{MR1191750} for related results. See also \cite{MR1714433} for a generalisation in the context of Gelfand pairs associated to $\mathbb H^n.$
  
 Extending and generalising the result in \cite{MR1269222}, we establish injectivity results for three different spherical means on an $H$-type group. In the remaining of this section we define these spherical means and state the injectivity results obtained.
 
 Let $G$ be an $H$-type group, identified with its Lie algebra $\mathfrak{g}$ via the exponential map. Then $\mathfrak{g}$ admits a an orthogonal decomposition $\mathfrak{g}=\mathfrak{v}\oplus \mathfrak{z}$, where $\mathfrak{z}$ is the center and $\mathfrak{v}$ its orthogonal complement. It is known that $\dim\mathfrak{v}$ has to be even, say $\dim\mathfrak{v}=2n$, and let $m$ denote $\dim\mathfrak{z}.$ We will identify $\mathfrak{v}$ with $\mathbb C^n$ and $\mathfrak{z}$ with $\mathbb R^m.$ This requires fixing an orthonormal basis on $\mathfrak{v}$ and $\mathfrak{z}.$ For most of our purposes, this can be an arbitrary chosen orthonormal basis, however for certain computations we will choose a basis with some properties (see \eqref{onb1}, \eqref{onb2}, \eqref{twisted}). We will write $(z,t)$ for points in $G$, where $z\in \mathbb{C}^n$ (identified with $\mathfrak{v}$) and $t\in\mathbb{R}^m$ (identified with $\mathfrak{z}$). The group law then is given by
 \[ (z,t)(w,s)=(z+w,t+s+\frac{1}{2}[z,w]),\] where $[\;,\;]$ denotes the Lie bracket. The Haar measure on $G$ is given by the Lebesgue measure on $\mathfrak{g}$ and will be denoted by $dz dt.$ Denote by $Q=2n+2m$ the homogeneous dimension of $G$.
 
 Next, we define three different spherical means and state the injectivity results. Since $m =1$ corresponds to the Heisenberg group, we will always assume that $m \geq 2$ unless explicitly stated. Let $\mu_r$ denote the normalized surface measure on the sphere $\{z\in \mathfrak{v}:|z|=r \}$ and consider the spherical means of a function $f,$
 \[ f\ast\mu_r(z,t)=\int\limits_{|w|=r}f(z-w,t-\frac{1}{2}[z,w])\, d\mu_r(w)\]
 
 \begin{thm}\label{firstinjectivity}
     Let $f\in L^p(G),\, 1\leq p\leq \frac{2m}{m-1}$. If $f\ast\mu_r(z,t)=0$ for all $(z,t)\in G,$ then $f$ vanishes identically.
 \end{thm}

Next, let $\nu_s, s > 0$ be the normalized surface measure on the sphere $\{y\in\mathfrak{z}:|y|=s \}$. Consider the measure $\mu_{r,s}=\mu_r\times\nu_s$. That is,
\[\int_G~ f(z,t)\, d\mu_{r,s}(z,t)=\int_G~ f(z,t)\, d\mu_r(z)d\nu_s(t).\]
Then, we define the {{\it bi-spherical means}} of $f$ by
\[ f\ast\mu_{r,s}(z,t)=\int\limits_{|w|=r}\int\limits_{|u|=s}f(z-w,t-u-\frac{1}{2}[z,w])\, d\mu_r(w)d\nu_s(u).\]
We have the following theorem.
\begin{thm}\label{secondinjectivity}
    Let $f\in L^p(G),\, 1\leq p\leq \frac{2m}{m-1}$. If $f\ast\mu_{r,s}(z,t)=0$ for all $(z,t)\in G$ then $f\equiv 0$.
\end{thm}

We also show by examples that the above two theorems are optimal. Finally, we define the {{\it homogeneous spherical means}}. Let $|(z,t)|$ denote a homogeneous norm on $G$ (see the next section for definition). There exists a unique Radon measure $\sigma$ on the unit sphere $\Sigma=\{(z,t):|(z,t)|=1 \}$ such that for all $f\in L^1(G)$
\[ \int\limits_G f(z, t)\, dz~dt =\int\limits_0^\infty\int\limits_\Sigma f(\delta_r(z,t))\, d\sigma(z,t)\, r^{Q-1}dr\] where $\delta_r$ denote the dilations that act as automorphisms of $G$ (see the next section for the definition). Dilating the measure $\sigma$ using $\delta_r$, for $r>0$ we can define $\sigma_r$  by
\[ \sigma_r(f)=\sigma(\delta_r f)=\int\limits_\Sigma f(\delta_r(z,t))\, d\sigma(z,t).\]
The homogeneous spherical means of a function $f$ is defined as the convolution $f \ast \sigma_r,$ of $f$ with $\sigma_r.$ For the homogeneous spherical means we have the following theorem.
\begin{thm}\label{thirdinjectivity} 
    \begin{enumerate}
        \item Let $m\geq 2$. If $f\in L^p(G),\, 1\leq p\leq \frac{2m}{m-1}$ and $f\ast\sigma_r(z,t)=0$ for all $(z,t)\in G$ then $f$ vanishes identically.
        \item Let $G=\mathbb{H}^n$, that is $m=1$, then the above injectivity holds for the range $1\leq p<\infty$.
    \end{enumerate}
\end{thm}

The plan of the paper is as follows: In the next section we recall all the required definitions and also state some known results that will be used later. In the third section we study the spectral decomposition of the sublaplacian of $G$ and prove the Abel summability. Using this, we prove the injectivity results in the final section.

\section{Preliminaries}
In this section, we recall some definitions and properties of $H$-type groups introduced by Kaplan \cite{MR554324}.
Let $\mathfrak{g}$ be a  finite dimensional real inner product space endowed with a Lie bracket that makes it into a two step nilpotent Lie algebra.  Let $\mathfrak{z}$ be its centre and $\mathfrak{v}$ be the  orthogonal complement of $\mathfrak{z}$. 
%Let $\mathfrak{g}$ be a two step nilpotent algebra endowed with an inner product $<,>$ and let $\mathfrak{v}$ be the orthogonal complement to its center $\mathfrak{z}$. 
For each $v\in \mathfrak{v}$, consider the map $ad_v:\mathfrak{v}\to \mathfrak{z}$ defined by 
\[ad_v(v') = [v, v'].\]
Let $\mathfrak{f}_v$ be the kernel of this map and  $\mathfrak{b}_v $   its orthogonal complement so that 
 \[\mathfrak{v}=\mathfrak{f}_v\oplus \mathfrak{b}_v .\] 
We shall say that $\mathfrak{g}$ is Heisenberg type or $H$-type if 
the map $ad_v$ is a surjective isometry for every unit vector $v\in \mathfrak{v}$. 
A connected and simply connected Lie group $G$ is of Heisenberg type if its Lie algebra is $H$-type.
For each non-zero $z\in \mathfrak{z}$ we can define the linear operator $J_z:\mathfrak{v}\to \mathfrak{v}$ by 
\[\langle J_z(v),v'\rangle=\langle z,[v,v']\rangle\qquad\mbox{ for all } v,v'\in\mathfrak{v}. \]
Then $J_z$ is a skew-symmetric linear isomorphism. Then $\mathfrak{g}$ is $H$-type if and only if %\cite{MR1102963}
\[J_z^2=-|z|^2 I.\]
This means that $J_z$ defines a complex structure on $\mathfrak{v}$ when $|z| = 1$ and therefore the dimension of $\mathfrak{v}$ is even. Hence, we identify  $\mathfrak{v}$ with $\mathbb{C}^n \equiv \mathbb{R}^{2n}$ and $\mathfrak{z}$ with $ \mathbb{R}^{m}$ for $n, m \in \mathbb{N}$.

The exponential map from $\mathfrak{g}$ to $G$ is a diffeomorphism. 
We can therefore  parametrise the elements of $G=\exp \mathfrak{g}$ by $(z,t)$, for $z$ in $\mathfrak{v} \equiv \mathbb{C}^n$ and $t$ in $\mathfrak{z} \equiv \mathbb{R}^m$. %Then the Haar measure on $G$ is $dg = dzdt$. 
By the Baker-Campbell-Hausdorff formula, it follows that the group law in $G$ is \[(z,t)(z',t')=(z+z', t+t'+\frac{1}{2}[z,z'])\quad \forall (z,t),(z',t')\in G. \]
Since $[\mathfrak{v}, \mathfrak{v}] \subset \mathfrak{z}$, the Lie bracket on $\mathfrak{v}$ can be written as (see \cite{MR2027639})  
\[[z,z']_j = \langle z,U^jz' \rangle\]
in terms of $2n\times2n$ skew-symmetric matrices  $U^j, j=1,2, \dots, m$. Since $J_z^2=-|z|^2 I$,   $U^j$ are orthogonal and satisfy
\[U^iU^j+U^jU^i=0,\; i \neq j.\]

The left invariant vector fields on $G$ which agree respectively with $\frac{\partial}{\partial x_j},\frac{\partial}{\partial y_j}$ at the origin are given by \[\begin{split}
    X_j&=\frac{\partial}{\partial x_j}+\frac{1}{2}\sum\limits_{k=1}^m\left(\sum\limits_{l=1}^{2n}z_lU_{l,j}^k \right)\frac{\partial}{\partial t_k},\\  Y_j& =\frac{\partial}{\partial y_j}+\frac{1}{2}\sum\limits_{k=1}^m\left(\sum\limits_{l=1}^{2n}z_lU_{l,j+n}^k \right)\frac{\partial}{\partial t_k},
\end{split} \] 
		where $z_l=x_l,z_{l+n}=y_l , l=1, 2,\ldots,n$.
	 The vector fields $T_k=\frac{\partial}{\partial t_k}, k=1,2,\ldots,m$ correspond to the centre of $\mathfrak{g}$.
Then the sublaplacian $\mathcal{L} = -\sum_j (X_j^2+Y_j^2) $ is given by \[\mathcal{L}=-\sum\limits_{j=1}^n(X_j^2+Y_j^2)=-\Delta_z+\frac{1}{4}|z|^2T-\sum\limits_{k=1}^m \langle z,U^k\nabla_z \rangle T_k, \] where \[\Delta_z=\sum\limits_{j=1}^{2n}\frac{\partial^2}{\partial z_j^2},\quad T=-\sum\limits_{k=1}^{m}\frac{\partial^2}{\partial t_k^2},\quad \nabla_z=\left(\frac{\partial}{\partial z_1},\frac{\partial}{\partial z_2},\cdots,\frac{\partial}{\partial z_{2n}} \right)^T. \]

For $a \in \mathbb R^m$ (identified with  $\mathfrak{z}^\ast$) let $f^a(z)$ stand for the inverse Fourier transform of the function $f(z, t)$ in the central variable. That is
\[ f^a(z) = \int_{\mathbb R^m}f(z, t)~e^{i \langle a, t \rangle}~dt.\]
For $a \neq 0$, let  $B(a)$ be the linear mapping on $\mathfrak{z}^{\perp}$ defined by
$$
\langle B(a) u, w\rangle=a([u, w]), \quad \text { for any } u, w \in \mathfrak{z}^{\perp}.
$$
Choose an orthonormal basis

\begin{equation}\label{onb1}
\left\{E_{1}(a), E_{2}(a), \cdots, E_{n}(a), \bar{E}_{1}(a), \bar{E}_{2}(a), \cdots, \bar{E}_{n}(a)\right\}
\end{equation}
of $\mathfrak{z}^{\perp}$ such that
$$
B(a) E_{i}(a)=-|a| \bar{E}_{i}(a), B(a) \bar{E}_{i}(a)=|a| E_{i}(a)
$$
and an orthonormal basis 
\begin{equation}\label{onb2}
\{ \epsilon_1, \epsilon_2, \cdots, \epsilon_m \}
\end{equation}
for $\mathfrak{z},$ such that $\langle a, \epsilon_1 \rangle = |a|$ and $\langle a, \epsilon_j \rangle = 0$ for $j =2,3, \cdots, m$. If $\mathfrak{g}$ is identified with $\mathbb{C}^n \times \mathbb{R}^m$ via this orthonormal basis, the first coordinate of the Lie bracket takes the form (see \cite{MR2801610})   
\[[z, z']_1 = \langle z,U^1z'\rangle =\sum\limits_{i=0}^n(x_i'y_i-y_i'x_i) = \Im(z \cdot \bar{z'}).\]
Hence the convolution with functions of the form  $g(z, t) = e^{-i\langle a,t \rangle}\varphi(z)$ can be written as
\begin{equation}\label{twisted}
    \begin{split}
f\ast g (z, t) &= \int_{\mathbb{C}^n}\int_{\mathbb{R}^m} f(z-w, t-s-\frac{1}{2}[z,w]) \varphi(w)e^{-i\langle a, s \rangle}\; dw ds\\
&=\int_{\mathbb{C}^n} f^a(z-w)\varphi(w) e^{-i\langle a,t \rangle}e^{\frac{i}{2}\langle a,[z,w] \rangle} \;dw\\
 &= e^{-i\langle a,t \rangle} f^a \times_{|a|} \varphi(z), 
 \end{split}
\end{equation} where the twisted convolution $\times_{|a|}$ of two suitable functions $f_1$ and $f_2$  on $\mathbb{C}^n$ is defined by 
\[ f_1 \times_{|a|} f_2 (z) = \int_{\mathbb C^n}~f_1(z-w)~f_2(w)~e^{\frac{i}{2}
|a| \Im\, z \cdot \overline{w}}~dw.\]
Also, one obtains the following result regarding the  action of the sublaplacian $\mathcal{L}$ on functions of the form $e^{-i \langle a, t \rangle} \varphi(z)$. 
\begin{lemma}
 Let $0 \neq a \in \mathfrak{z}^{*}$. If $f(z, t)=e^{-i\langle a, t\rangle} \varphi(z)$, then
\[
\mathcal{L} f(z, t)=e^{-i(a, t\rangle} L_{|a|} \varphi(z)
\]
where, for $\lambda > 0$ \[L_\lambda=-\Delta_z+\frac{\lambda^2|z|^2}{4}-i\lambda\sum\limits_{j=1}^{n}\left(x_j\frac{\partial}{\partial y_j}-y_j\frac{\partial}{\partial x_j}\right)\] 
is the twisted Laplacian on $\mathbb{C}^n$.
\end{lemma}

For a proof, see Lemma 1 in \cite{MR2801610}. 
An $H$-type group admits a family of dilations which act as automorphisms of $G$ by
\[\delta_r(z, t) = (rz, r^2t), r> 0.\]
It is easy to see that $G$ with this family of dilations is a homogeneous Lie group whose  homogeneous dimension  is  $2n+2m$ which we denote by $Q$ (see \cite{MR657581}). The Kor\'{a}nyi norm on $G$ is defined as 
\[|(z,t)| = \left(|z|^4+|t|^2\right)^{1/4}.\]
It is clear that $|\delta_r(z,t) = r |(z,t)|$.

A smooth kernel $K$ on $G \setminus \{0\}$ is said to be homogeneous of degree $-Q$ if $$K( \delta_r(z,t) ) = r^{-Q} K(z,t), \forall (z,t) \in G \setminus \{0\}.$$ 
Homogeneous kernels which satisfy a cancellation condition (see below) defines singular integral operators on $G$. The cancellation condition is given by 

\begin{equation}\label{meanzero}
\int\limits_{a < |(z,t)| <b} K(z,t)\; dzdt = 0,  \forall \;   0 < a < b < \infty.  
\end{equation}
Notice that, since $\{ (z, t) : a < |(z,t)| < b\}$ is relatively compact, the above integral is well defined. For more details on such operators, we refer to \cite{MR657581}. Now we collect some of the results about singular integral operators on $G$ which will be used later.
\begin{thm}\label{cancel}
    Let $G$ be a connected, simply connected $H$-type group. Let $K \in C^{\infty}(G \setminus \{0\})$ be a kernel which is homogeneous of degree $-Q$. Assume that $K$ satisfies the cancellation condition \[\int\limits_{a < |(z,t)| <b} K(z,t)\; dzdt = 0,  \forall \;   0 < a < b < \infty.  \]
    Then the singular integral operator 
    \[f \mapsto f\ast K\]
    is bounded on $L^2(G)$.
\end{thm}
\begin{proof}
    This is a special case of Theorem 1 in \cite{MR582703}*{p. 494}.
\end{proof}

The next theorem says that for the above operators, the $L^2$-boundedness imply the $L^p$-boundedness.
\begin{thm}\label{CZp}
       Let $G$ be an $H$-type group and $K \in C^{\infty}(G \setminus \{0\})$ be a kernel that satisfy the cancellation condition and  is homogeneous of degree $-Q$. If the operator $f \mapsto f\ast K$ is bounded on $L^2(G)$, then it is bounded on $L^p(G)$ for $1 < p< \infty$.
\end{thm}
\begin{proof}
    Follows from Theorem 5.1 of \cite{MR463513}.
\end{proof}
We end this section by restating the cancellation condition. 
\begin{lemma}\label{cancellemma}
    Let  $K \in C^{\infty}(G \setminus \{0\})$ be homogeneous of degree $-Q$. Then the cancellation condition in \eqref{meanzero} is equivalent to the condition
    \[\int_{\mathbb{C}^n}\int_{S^{m-1}} K(z, u) \; dz d\nu(u) = 0\]
    where $\nu$ is the normalised surface measure on the unit sphere in $\mathfrak{z}.$ In particular, if $K$ is radial in the $t$-variable, the cancellation condition is equivalent to 
    \[\int_{\mathbb{C}^n} K(z, 1)\; dz = 0.\]
\end{lemma}
\begin{proof}
    
    When $K$ is homogeneous of degree $-Q$
    \begin{align*}
        \int\limits_{a < |(z,t)| <b} K(z,t)\; dzdt &= \int\limits_{\mathbb{C}^n} \int\limits_{S^{m-1}}\int\limits_{ a^4 < |z|^4 + s^2 < b^4} K(z, su)  s^{m-1} ds d\nu(u) dz\\
        &= \int\limits_{\mathbb{C}^n} \int\limits_{S^{m-1}}\int\limits_{ a^4 < |z|^4 + s^2 < b^4} K\left(\frac{z}{\sqrt{s}}, u\right)  s^{-n-1} ds d\nu(u) dz\\
        &= \int\limits_{\mathbb{C}^n} \int\limits_{S^{m-1}}\int\limits_{ a^4 < s^2(1+|w|^4) < b^4}  \frac{ds}{s} K(w, u)  d\nu(u) dz.
    \end{align*}
    Now the result follows from the fact that 
    \[\int\limits_{ a^4 < s^2(1+|w|^4) < b^4}  \frac{ds}{s} = \int_{\frac{a^2}{\sqrt{1+|w|^4}}}^{\frac{b^2}{\sqrt{1+|w|^4}}}\frac{ds}{s}  = \log(\frac{b^2}{a^2})\]
    is independent of $w$.
\end{proof}
    
We need the following result which is a special case of a result due to Christ (see \cite{MR819558}*{p. 575}.

\begin{thm}\label{christ}
Let $G$ be an $H$-type group, with dilations $\{\delta_t:~ t> 0\}.$ Let $\gamma: \mathbb R \to G$ be an odd homogeneous curve, that is $\gamma(t) =\exp(\delta_t(Y_+))$ for $t> 0$ and $\gamma(t) = \exp(\delta_{-t}(Y_-))$ where $Y_+ = - Y_- \in \mathfrak{g},$ so that $\gamma(t) = -\gamma(-t) = \gamma(t^{-1}).$ Then, the operator $$H_\gamma f(x) = \operatorname{p.\!v.} \int_{\mathbb R}~f(x \cdot \gamma(t)^{-1})~\frac{dt}{t},$$ is bounded on $L^p(G)$ for $1 < p < \infty$ with norm independent of the curve $\gamma.$
 
 \end{thm}

We shall also need the following result connecting the $L^p$ membership of a function on $\mathbb R^m$ with the dimension of the support of the Fourier transform of the function.

\begin{thm}\label{na}
 Let $f \in L^p(\mathbb R^m)$ and support of $\widehat{f}$ (distributional Fourier transform of $f$) is contained in a $C^1$-manifold of dimension $0< d < m.$ Then $f$ vanishes identically provided $1 \leq p \leq \frac{2m}{d}.$ If $d = 0,$ $f$ vanishes identically provided $1 \leq p < \infty.$ 

\end{thm}

\begin{proof}
When the support is a sphere, this follows from \cite{MR1269222} (see Lemma 2.2 and Theorem 2.2 there). For the general case see \cite{MR2066426} (Theorem 1).
\end{proof}
\section{Spectral projections and Abel summability}

In this section we prove a summability result for the spectral decomposition of the sublaplacian on $L^p$  for $2 \leq p < \infty.$ We follow the methods in \cite{MR1101262}.
For $a \in \mathbb{R}^m$, recall that  \[e_k^a(z,t)=e^{-i\langle a,t \rangle} \varphi_k^{|a|}(z),\] where the scaled Laguerre functions $\varphi_k^\lambda$ for $\lambda > 0,$ defined by
\[\varphi_k^\lambda(z)=L_k^{n-1}\left(\frac{\lambda|z|^2}{2}\right)e^{-\frac{1}{4}\lambda|z|^2}, k=0, 1, 2,\ldots\] in terms of the Laguerre polynomials $L_k^{n-1},$ are the  eigenfunctions of the twisted Laplacian $L_\lambda$ with eigenvalue $(2k+n)|\lambda|$. 
Hence, 
\[\mathcal{L} e_k^a(z,t) = e^{-i\langle a,t \rangle} L_{|a|}\; \varphi_k^{|a|}(z) = (2k+n)|a| e_k^a(z,t).\]

Next we explain the $L^2$ spectral decomposition. %***Using the, we have that $$ f \ast e_k^a(z, t) = e^{i \langle a, t \rangle} f^a \times_{|a|} \varphi_k^{|a|}(z).$$
Applying the Fourier inversion formula in the central variable and using the special Hermite expansion of a function on $\mathbb C^n,$ we obtain
\begin{equation*}
    \begin{split}
    f(z,t)     & =\frac{1}{(2\pi)^m}\int_{\mathbb{R}^m}f^a(z)e^{-i \langle a,t \rangle}\, da \\
         & =\frac{1}{(2\pi)^m}\int_{\mathbb{R}^m}\frac{|a|^n}{(2\pi)^n}\sum\limits_{k=0}^\infty(f^a\times_{|a|}\varphi_k^{|a|}(z))e^{-i \langle a,t \rangle}\, da.\\
         & =\frac{1}{(2\pi)^{n+m}}\int_{\mathbb{R}^m}\sum\limits_{k=0}^\infty f\ast e_k^a(z,t))~|a|^n\, da\\
         & = \frac{1}{(2\pi)^{n+m}} \sum_{k=0}^\infty~\int_{\mathbb{R}^m}~f \ast e_k^a(z, t)~|a|^n~da.
    \end{split}
\end{equation*}
The above gives the $L^2$ spectral decomposition of $f.$  We also have, by the Plancherel formula (see \cite{MR2801610}*{p. 2717},
$$\|f\|_{L^2(G)} = \frac{1}{(2\pi)^{n+m}}~\sum_{k = 0}^\infty~\int_{\mathbb R^m}~|a|^{2n}~\int_{\mathbb C^n}~|f \ast e_k^a (z, 0)|^2~dz~da. $$
%On the formal level this is convolution with the kernel \[ K(z,t)=\frac{1}{(2\pi)^{n+m}}\int\limits_{\mathbb{R}^m} \sum\limits_{k=0}^\infty e_k^a(z,t)|a|^n\, da. \] Now let us consider the projection operator $P_k:f\mapsto f*A_k$ where 
Let $\mathcal{A}_k$ denote the spectral projection operator defined by 
\begin{equation}\label{spectral-projection-def}
 \mathcal{A}_k f(z, t) = \int_{\mathbb{R}^m} ~f\ast e_k^a(z,t)~|a|^n~da. 
\end{equation}

Our aim is to write the spectral projection operator $\mathcal{A}_k$ as a convolution operator and prove its $L^p$ boundedness. To write this operator as a convolution operator, we define the kernel,

\begin{align*}  
 A_k(z,t) &=\int\limits_{\mathbb{R}^m} ~e_k^a(z,t)~|a|^n~da \\ 
 &= \int\limits_{\mathbb{R}^m}~ e^{-i \langle a, t \rangle}~ \varphi_k^{|a|}(z)|a|^n~ da.
 \end{align*}
  Since $\varphi_k^{|a|}(z) = L_k^{n-1}\left(\frac{|a||z|^2}{2}\right) e^{-\frac{|a|}{4}|z|^2}$, the kernel $A_k(z,t)$ is a linear combination of functions of the form 
 \[ A_k^j(z,t) = |z|^{2j} \int\limits_{\mathbb{R}^m}~ e^{-i \langle a, t \rangle}~ e^{-\frac{|a|}{4}|z|^2} ~|a|^{n+j}\, da,\; j= 0, 1, \cdots, k.\]
 A simple change of variables shows that 
 \[A_k^j(sz,s^2t) = s^{-Q} A^j_k(z,t), \] which is the required homogeneity for singular integral operators on $G.$

 Using polar coordinates, we obtain
 \[A_k^j(z,t) = c_m |z|^{2j}\int_0^{\infty}\frac{J_{\frac{m}{2}-1}(\lambda |t|)}{(\lambda |t|)^{\frac{m}{2}-1}}e^{-\frac{\lambda}{4}|z|^2} \lambda^{n+m+j-1}\, d\lambda,  \]
 where $c_m$ is a constant that depends only on $m$.
 We prove that $A_k(z,t)$ is a Calder\'on-Zygmund kernel by showing that each $A_k^j(z,t)$ is. Since $A_k^j(z,t)$ is homogeneous of degree $-Q$ and belongs to $C^{\infty}(G \setminus \{0\})$, we need to show that these kernels satisfy the cancellation condition as in Lemma \ref{cancellemma}. Since  $A_k^j(z,t)$ is radial in $t$, it suffices to show the following:
 
 \begin{lemma}\label{intak}
     \[\int_{\mathbb{C}^n} A_k^j(z,1) = 0,~ j = 0, 1, 2, \dots, k.\]
 \end{lemma}
 
 \begin{proof}
     We start with the integral 
\begin{equation}\label{im}
     I_m (\tau) = \int_0^{\infty}~ \frac{J_{\frac{m}{2}-1}(\lambda)}{\lambda^{\frac{m}{2}-1}} ~e^{-\tau\lambda} ~\lambda^{m-1}\; d\lambda, ~\tau > 0.
\end{equation}
     Then for any $t \in \mathbb{R}^m$ such that $|t|= 1,$ it is easy to see that (up to a constant)
     $$
          I_m (\tau) = \int_{\mathbb{R}^m}~e^{-i \langle x, t \rangle } ~e^{-\tau|x|}~dx. $$ The above equals the Poisson kernel, 
 $$ c_m \frac{\tau}{(1+\tau^2)^{\frac{m+1}{2}}}$$ for some constant $c_m$.

 Now,
      \begin{align*}
               \int_0^{\infty}~ \frac{J_{\frac{m}{2}-1}(\lambda)}{\lambda^{\frac{m}{2}-1}} ~e^{-\tau\lambda}~ \lambda^{n+m+j-1}~ d\lambda &= \frac{d^{n+j}}{d\tau^{n+j}}\left(I_m (\tau)\right) \\
     &= I_m^{(n+j)}(\tau).
       \end{align*}
      Hence, to prove the lemma, we need to show that 
      \[\int_{\mathbb{C}^n}~|z|^{2j}~ I_m^{(n+j)}(|z|^2)\; dz = 0, ~j = 0, 1, 2, \dots, k.\]
     Since the integrand is radial, this reduces to showing that
      \[ \int_0^{\infty}~ I_m^{(n+j)}(r^2) ~r^{2n+2j-1}~ dr = \frac{1}{2}\int_0^{\infty} ~I_m^{(n+j)}(b) ~b^{n+j-1} \; db = 0.\]
      Now, writing \[\Psi(b) = \frac{1}{(1+b^2)^{\frac{m+1}{2}}}, \]
      we get,
      \[I_m^{(n+j)}(b) = b \Psi^{(n+j)}(b) + (n+j)\Psi^{(n+j-1)}(b).\]
      Hence 
      \begin{align*}
          \int_0^{\infty}~ I_m^{(n+j)}(b) ~b^{n+j-1}\; db &= \int_0^{\infty}~ \Psi^{(n+j)}(b)~ b^{n+j}\; db + (n+j) \int_0^{\infty}~ \Psi^{(n+j-1)}(b)~b^{n+j-1}\; db\\
          &= \lim\limits_{b \to \infty} b^{n+j} \Psi^{(n+j-1)}(b)
      \end{align*}
      which is easily verified to be zero as $m \geq 2$. This proves the lemma.
 \end{proof}
 From Theorem \ref{cancel}, it follows that the operator
 \[ f \mapsto f \ast A_k\]
 is a bounded operator on $L^2(G)$ and therefore (by Theorem \ref{CZp}) bounded  on $L^p(G), 1 < p < \infty$ as well.
 
 Now for  $f\in\mathscr{S}(G),$ the Schwartz space of $G,$ with $f^a$ compactly supported in the variable $a$, it is easy to check that 
 \[\mathcal{A}_kf(z,t) = f \ast A_k(z,t).\]
 It follows that 
  \[\mathcal{A}_kf(z,t) = f \ast A_k(z,t),\quad \forall f \in L^2(G)\]
  and $\mathcal{A}_kf = f \ast A_k$ extends to a bounded operator on $L^p(G), 1 < p < \infty$.
  Hence, we have proved the following theorem.
  \begin{thm}\label{spectraloper}
     The spectral projection operator $\mathcal{A}_k$ is the convolution operator  $ f \mapsto f \ast A_k$ and is bounded on $L^p(G), 1 < p < \infty$.
  \end{thm}
  
  Next we show the Abel summability of the spectral decomposition for $f \in L^p(G)$.
  \begin{thm}\label{abelsum}
      Let $ 2 \leq p < \infty$ and $f \in L^p(G)$. Then
      \[\lim\limits_{r \to 1} \sum\limits_{k=0}^{\infty} r^k \int_{\mathbb{R}^m}f \ast e_k^a(z,t)~ |a|^n~ da = f(z, t) \]
      in the $L^p$ norm.
  \end{thm}
As in \cite{MR1101262} (see Theorem 3.3 and Corollary 3.4 there, also the first paragraph in Page 375), it is enough show that the operators
\begin{equation}
 T_rf(z,t) =  \sum\limits_{k=0}^{\infty} r^k \int_{\mathbb{R}^m}~f \ast e_k^a(z,t) ~|a|^n~ da, 
 \end{equation}
 are uniformly bounded on $L^p(G),$  $2 \leq p < \infty$.  To prove this, we need the following lemma.

\begin{lemma}\label{lm3}
    Let $K(z,t)$ be an odd kernel which is homogeneous of degree $Q$. Then the operator norm of $ f \mapsto f\ast K$ on $L^p(G), 1 < p < \infty$ is bounded by \[C_p\int_{\mathbb{C}^n}\int_{S^{m-1}}|K(w,u)|\, dwdu\]
for some constant $C_p$ depending only on $p$.\end{lemma}
\begin{proof}
Using the homogeneity of the kernel $K$, we can write $f \ast K$ as,
         \begin{align*}
            f\ast K(z,t) & = \int\limits_{\mathbb{C}^n}\int\limits_{\mathbb{R}^m}f(z-w,t-s-\frac{1}{2}[z,w])K(w,s)\, dwds\\
            & = \int\limits_{\mathbb{C}^n}\int\limits_0^\infty\int\limits_{S^{m-1}}f(z-w,t-ru-\frac{1}{2}[z,w])r^{-n-m}K\left(\frac{w}{\sqrt{r}},u\right)r^{m-1}\,drdudw\\
            & = \int\limits_0^\infty\int\limits_{S^{m-1}}\int\limits_{\mathbb{C}^n}f(z-\sqrt{r}w,t-ru-\frac{\sqrt{r}}{2}[z,w])K(w,u)r^{-1}\,dwdrdu\\
            & = 2\int\limits_{\mathbb{C}^n}\int\limits_{S^{m-1}}\int\limits_0^\infty f(z-rw,t-r^2u-\frac{r}{2}[z,w])K(w,u)\, \frac{dr}{r}dudw.\\
                 \end{align*}
          Since $K$ is odd, the above integral becomes
          
    \begin{multline*}
        2\int\limits_{\mathbb{C}^n}\int\limits_{S^{m-1}}\left(\int\limits_0^\infty f(z-rw,t-r^2u-\frac{r}{2}[z,w])\, \frac{dr}{r}\right.\\ -\left. \int\limits_0^\infty f(z+rw,t+r^2u+\frac{r}{2}[z,w])\, \frac{dr}{r}\right) K(w,u)\,dudw
    \end{multline*}

Now, the inner integral, 

   $$ \int\limits_0^\infty ~f(z-rw,t-r^2u-\frac{r}{2}[z,w])\, \frac{dr}{r}- \int\limits_0^\infty~ f(z+rw,t+r^2u+\frac{r}{2}[z,w])\, \frac{dr}{r}$$
   equals 
\begin{equation}\label{ht}
   \int\limits_0^\infty~ f((z,t)(\delta_r(w,u))^{-1})\, \frac{dr}{r}- \int\limits_0^\infty~ f((z,t)(\delta_r(w,u)))\, \frac{dr}{r}
\end{equation}
since 
\begin{equation*}
    \begin{split}
        %\int\limits_{-\infty}^0f((z,t)(\gamma_{(w,s)}(r))^{-1})\, \frac{dr}{r} & =
        \int\limits_{-\infty}^0f((z,t)(\delta_{-r}(-w,-s))^{-1})\, \frac{dr}{r} & =
         \int\limits_{-\infty}^0 f((z,t)\delta_{-r}(w,s))\, \frac{dr}{r}\\
        & = -\int\limits_0^\infty f((z,t)\delta_{r}(w,s))\, \frac{dr}{r}.\\
    \end{split}
\end{equation*}

The expression \eqref{ht} is the Hilbert transform $H_{\gamma_{(w,s)}}f$ of $f$ along the curve $\gamma_{(w,s)}$ in $G$, given by,
\begin{equation*}
    \gamma_{(w,s)}(r)=\left\{\begin{array}{cc}
    \delta_r(w,s)     & \mbox{for }r>0  \\
    \delta_{-r}(-w,-s)     &\mbox{for }r\leq 0 
    \end{array}\right.
\end{equation*}

Hence 
\[f\ast K(z,t) =\int\limits_{\mathbb{C}^n}\int\limits_{S^{m-1}}H_{\gamma_{(w,s)}}f(z,t)K(w,s)\, dwds.\]
Therefore,
 \begin{equation*}
    \begin{split}
        \|f\ast K\|_p & \leq \int\limits_{\mathbb{C}^n}\int\limits_{s^{m-1}}\|H_{\gamma_{(w,s)}}f\|_p|K(w,s)|\, dwds\\
        & \leq C_p\|f\|_p\int\limits_{\mathbb{C}^n}\int\limits_{s^{m-1}}|K(w,s)|\, dwds\\
    \end{split}
\end{equation*}
where $C_p$ is a constant that depends only on $p$ (by Theorem \ref{christ}).
\end{proof}

Now to prove the uniform boundedness of $\|T_r\|_p$ using the previous lemma, we note that $T_r$ is a convolution operator with kernel
\[\sum\limits_{k=0}^{\infty} r^k \int_{\mathbb{R}^m} e_k^a(z,t) |a|^n\; da  \]
which we compute using the following generating function identity of Laguerre polynomials.
\[ \sum\limits_{k=0}^{\infty} r^k L_k^{\alpha}(x) = (1-r)^{-\alpha-1} e^{-\frac{rx}{1-r}},\; |r| < 1.\]
It then follows that
\[\sum\limits_{k=0}^{\infty} r^k \int_{\mathbb{R}^m} e_k^a(z,t) |a|^n\; da = (1-r)^{-n} \int_{\mathbb{R}^m} e^{-i \langle a, t\rangle} e^{-\frac{1}{4}\frac{1+r}{1-r}|a||z|^2} |a|^n \; da. \]
Since this is not an odd kernel, we bring in the Riesz transform in the $t$- variable. Define the operator $\mathcal{R}_j$ by 
$$(\mathcal{R}_jf )^a(z) = \frac{a_j}{|a|}~f^a(z),  $$
which is just the $j$-th Riesz transform in the central variable. Clearly, $\mathcal{R}_j$ is bounded on $L^p(G)$ for $1 < p < \infty.$ Now, define the operator
 
\[\mathcal{R}_j\mathcal{A}_kf(z,t) = \int\limits_{\mathbb{R}^m} f \ast  e_k^a(z,t) \frac{a_j}{|a|}|a|^n \; da.\] 
%As $\mathcal{A}_k$ is $L^p$ bounded, $\mathcal{R}_j\mathcal{A}_k$ is also bounded on $L^p(G)$ %for $1 < p < \infty.$
%Notice that $\mathcal{R}_j$ is a bounded operator on $L^p(G)$ for $1 < p < \infty,$ as $$ %\mathcal({R}_jf )^a (z) = \frac{a_j}{|a|} f^a(z) $$ and the
Since $\sum\limits_{j = 1}^m \mathcal{R}_j^2 = I$, it suffices to prove that the operator norm of $\sum\limits_{k=0}^{\infty}r^k \mathcal{R}_j\mathcal{A}_k$ is indepedent of $r.$
Now the kernel  of the above operator is 
\[\sum\limits_{k=0}^{\infty} r^k \int_{\mathbb{R}^m}~ e_k^a(z,t) \frac{a_j}{|a|}|a|^n\; da = (1-r)^{-n} \int_{\mathbb{R}^m} ~e^{-i \langle a, t\rangle} ~e^{-\frac{1}{4}\frac{1+r}{1-r}|a||z|^2} ~\frac{a_j}{|a|}~ |a|^n \; da.  \]
Writing in terms of the polar coordinates and using the Hecke-Bochner identity, we obtain that the above integral is a constant multiple of 
\[(1-r)^{-n}~t_j~\int_0^{\infty}\frac{J_{\frac{m}{2}}(\lambda|t|)}{(\lambda|t|)^{\frac{m}{2}}}~ e^{-\frac{1}{4}\frac{1+r}{1-r}\lambda|z|^2}~\lambda^{n+m-1}\; d\lambda. \]
When $t \in S^{m-1}$, we can write the above expression using the function $I_m$ (see \eqref{im} )  as
\[ (1-r)^{-n} t_j I_{m+2}^{(n-2)}\left(-\frac{1}{4}\frac{1+r}{1-r}|z|^2\right). \]
Since \[\int_{\mathbb{C}^n} \int_{S^{m-1}} \left|(1-r)^{-n} t_j I_{m+2}^{(n-2)}\left(-\frac{1}{4}\frac{1+r}{1-r}|z|^2\right)\right| \; dz dt \leq  C \frac{1}{(1+r)^n}  \int_0^{\infty} |I_{m+2}^{(n-2)}(a)|a^{n-1}\; da, \]
the proof is complete as it can  easily be verified that
\[\int_0^{\infty} |I_{m+2}^{(n-2)}(a)|a^{n-1}\; da \leq C.\]

\begin{remark}
 It is worth noting that, in the above proofs the explicit expression of the kernels were not used (unlike the arguments in \cite{MR1101262}). Using the generating function for the Laguerre polynomials, it is possible to obtain an expression for the kernel of the spectral projections $\mathcal{A}_k.$ Indeed, $$ \sum_{k=0}^\infty~r^k~\int_{\mathbb R^m}~e_k^a(z, t)~|a|^n~da = (1-r)^{-n}~\int_{\mathbb{R}^m} ~e^{-i \langle a, t\rangle} ~e^{-\frac{1}{4}\frac{1+r}{1-r}|a||z|^2}~ |a|^n \; da.$$ Using polar coordinates in the above leads to the expression (up to a constant) $$ (1-r)^{-n}~\int_0^{\infty}\frac{J_{\frac{m}{2}-1}(\lambda|t|)}{(\lambda|t|)^{\frac{m}{2}-1}}~ e^{-\frac{1}{4}\frac{1+r}{1-r}\lambda|z|^2}~\lambda^{n+m-1}\; d\lambda. $$  Substituting the well known integral formula for the Bessel function in the above, we get (again up to a constant)
 $$\int_{-1}^{1}~(1-s^2)^{\frac{m-3}{2}}~\left ( (1-r)^{-n}~\int_0^\infty~e^{is\lambda |t|}~e^{-\frac{1}{4}\frac{1+r}{1-r}\lambda|z|^2}~\lambda^{n+m-1}\; d\lambda \right )~ds.$$ Now, the kernel $A_k$ is the $k^{\text{th}}$-derivative of the above with respect to $r,$ evaluated at $ r = 0.$ However, the inner integral can be computed as in \cite{MR1101262}*{p. 362}. We obtain that the expression $$  (1-r)^{-n}~\int_0^\infty~e^{is\lambda |t|}~e^{-\frac{1}{4}\frac{1+r}{1-r}\lambda|z|^2}~\lambda^{n+m-1}\; d\lambda$$ equals (ignoring some constants that depend only on $n$ and $m$)
 $$ (1-r)^{m}~\left [ (|z|^2-4is|t|) + r (|z|^2+4is|t|) \right ]^{-n-m}.$$
 Differentiating the above $k$ times and evaluating at $r = 0,$ we obtain that 
 \begin{equation}\label{spec-expression}
A_k(z, t) = c_{n, m}~\int_{-1}^{1}~(1-s^2)^{\frac{m-3}{2}}~\left [ 
\sum_{\ell =0}^{\text{min}\{k, m-1\}}~(-1)^{\ell}~\binom{m-1}{\ell}~P_{k-\ell}(z, s|t|)
\right ]~ds
 \end{equation}
 where $c_{n, m}$ is a constant depending only on $n$ and $m$ and 
 $$P_j(z, s|t|) = \frac{(n+m-1+j)!}{j!}~\frac{(|z|^2-4is|t|)^j}{(|z|^2+4is|t|)^{n+m-1+j}}
 \left (1 + \frac{j}{n+m-1+j}~\frac{(|z|^2+4is|t|)}{(|z|^2-4is|t|)} \right ).$$
 When $m$ is odd, $p = \frac{m-3}{2}$ is a non-negative integer and one can expand the term $(1-s^2)^p$ in \eqref{spec-expression} and prove the cancellation condition for the kernel $A_k$ by a somewhat long induction argument. However, this does not seem to work when $m$ is even.

\end{remark}
\section{Injectivity of spherical means}

In this section we prove the theorems stated in the introduction. We follow the proofs given in \cite{MR1269222} closely. The important point is that the function $e_k^a(z, t)$ are eigenfunctions for the three spherical mean operators we have considered.

First we look at the spherical means with respect, to the  the normalized surface measure $\mu_r$ on the sphere $\{z\in \mathfrak{v}:|z|=r \}$. As in   (\ref{twisted}), we can see that
\[ e_k^a \ast \mu_r(z,t) = e^{-i \langle a, t \rangle} \varphi_k^{|a|} \times_{|a|} \mu_r(z). \]
Since (see \cite{MR1269222}) \[\varphi_k^{|a|} \times_{|a|} \mu_r(z) = \frac{k! (n-1)!}{(k+n-1)!} \varphi_k^{|a|}(r) \varphi_k^{|a|} (z)\]
we obtain,
\begin{equation}\label{eigen1}
    e_k^a\ast \mu_r(z,t)=~c_{k, n}~\varphi_k^{|a|}(r)~e_k^a(z,t),~\forall~(z, t) \in G,
\end{equation}
where $c_{k, n} = \frac{k! (n-1)!}{(k+n-1)!}.$
    Now, let $f\in L^p(G), 1\leq p \leq \frac{2m}{m-1}$ and assume that $f \ast \mu_r$ vanishes identically. Convolving $f$ with a smooth approximate identity, we may assume that $f\in L^p$ for $2 \leq p \leq \frac{2m}{m-1}$. From the above identity \eqref{eigen1}, the spectral decomposition of $f\ast \mu_r$ is given by
    \[f\ast \mu_r(z,t)=\sum\limits_{k=0}^{\infty} \int\limits_{\mathbb{R}^m}~c_{k, n}~ \varphi_k^{|a|}(r)~f\ast e_k^a(z,t)~|a|^n\, da .\]
    If $f\ast \mu_r(z,t)=0$ for all $(z,t)$, by Theorem \ref{abelsum},
    \[\lim\limits_{s\to 1}\sum\limits_{k=0}^\infty~ c_{k, n}~s^k~ \int\limits_{\mathbb{R}^m}~\varphi_k^{|a|}(r)~f\ast e_k^a(z,t)~|a|^n\, da=0 \]
    where the convergence is in $L^p(G)$. Applying the $k^{\text{th}}$ spectral projection operator $\mathcal{A}_k$ and using Theorem \ref{spectraloper} we obtain that
    \[\int\limits_{\mathbb{R}^m}\varphi_k^{|a|}(r) \left(f^a \times_{|a|} \varphi_k^{|a|}\right)(z)~ e^{-i\langle a, t \rangle} |a|^n\, da=0, \quad \forall (z,t), \forall k=0,1,2,\ldots \]
    Arguing as in \cite{MR1269222}*{p.276} (also see \cite{MR1714433}*{pp.257-258}), we obtain that, for almost all $z\in \mathbb{C}^n,$ the  support of $f^a \times_{|a|}\varphi_k^{|a|}(z)$, the distributional Fourier transform of $\mathcal{A}_kf(z,\cdot)$, is contained in the zero set of $L_k^{n-1}(\frac{1}{2}|a|r^2)$, which is a finite union of spheres in $\mathbb R^m.$ But this implies, by Theorem \ref{na}, that  $\mathcal{A}_kf(z,t)$ is zero as $\mathcal{A}_kf\in L^p$ for $1 < p \leq \frac{2m}{m-1}$. This finishes the proof of Theorem \ref{firstinjectivity}.

Next, we show that the above range is optimal by an example. For a fixed $k \geq 1$ and $s > 0,$ let 
\[\begin{split}
    F(z,t) & = \frac{J_{\frac{m}{2}-1}(s|t|)}{(s|t|)^{\frac{m}{2}-1}}~\varphi_k^{s}(z) \\%\varphi_k^s(z)
    &= \int\limits_{|a|=s}~e^{-i \langle a,t \rangle}~\varphi_k^{|a|}(z)~ d\sigma_s(a)\\
    & = \int\limits_{|a| =s}~e_k^a(z, t)~d\sigma_s(a),
\end{split}\]
where $\sigma_s$ is the normalized surface measure on the sphere $\{a \in \mathbb R^m:~|a| = s\}.$
An easy computation using \eqref{eigen1} shows that,
\[ F \ast \mu_r(z,t)=~c_{k, n}~\varphi_k^s(r)~F(z,t), \]
for all $(z, t) \in G.$
Choosing $s$ suitably, we can make sure that $\varphi_k^s(r)=0$. From the asymptotics of the Bessel function it is clear that $F \in L^p(G)$ if and only if $p>\frac{2m}{m-1}$, which proves our claim.

Now we look at the  bi-spherical means defined using the measures $\mu_{r,s}=\mu_r\times\nu_s$.
Recall that the measure $\mu_{r, s}$ for $r >0, s >0$ was defined by 
$$ \mu_{r, s}(f) = \int_{|z|=r}~\int_{|t| =s}~f(z, t)~d\mu_r(z)~d\nu_s(t),$$ where $d\mu_r$ and $d\nu_s$ are normalized surface measures on the spheres $\{z:~|z| =r\}$ and $\{t:~|t| = s\}$ respectively. 
Assume that $f \in L^p(G)$ for $2 \leq  p \leq \frac{2m}{m-1}$ and $f \ast \mu_{r, s}$ vanishes identically. Proceeding as in the earlier proof, using the identity \eqref{eigen1}, we get
    \[\begin{split}
        e_k^a\ast \mu_{r,s}(z,t) & = c_{k, n}~e^{-i\langle a,t \rangle}~\frac{J_{\frac{m}{2}-1}(s|a|)}{(s|a|)^{\frac{m}{2}-1}}~\varphi_k^{|a|}(r)~\varphi_k^{|a|}(z)\\
        & = ~c_{k, n}~\frac{J_{\frac{m}{2}-1}(s|a|)}{(s|a|)^{\frac{m}{2}-1}}~\varphi_k^{|a|}(r)~e_k^{a}(z,t).
    \end{split}\]
 Continuing exactly as above we get that the distributional Fourier transform of $\mathcal{A}_kf(z,t)$ in the $t$ variable is supported in the zero set of (as a function of $a$) \[\frac{J_{\frac{m}{2}-1}(s|a|)}{(s|a|)^{\frac{m}{2}-1}}~\varphi_k^{|a|}(r),\] which is a union of infinitely many spheres in $\mathbb R^m.$ It then follows that $\mathcal{A}_kf=0$ from Theorem \ref{na}, if $1\leq p\leq \frac{2m}{m-1}$.
This completes the proof of Theorem \ref{secondinjectivity}.

Next we show that the above range is the best possible. To this end, we need to recall some results on bi-radial functions on an $H$-type group $G$. Define the averaging operator (see \cite{MR1703838}*{p. 221} $\Pi$ on integrable functions on $G$ by

$$ \Pi(f)(z, t) = \int_{S^{m-1}}~\int_{S^{2n-1}}~f(|z|u, |t|v)~d\mu(z)~d\nu(v),$$
where $d\mu$ and $d\nu$ are the normalized surface measures on the unit spheres $\{z:~|z| =1\}$ and $\{t:~|t|=1\}$ respectively. The operator $\Pi$ is then an averaging projector satisfying several properties (see \cite{MR1703838}*{p. 220}.  

A bi-radial function on $G$ is a function $f$ that satisfies $\Pi(f) =f.$ Clearly, $f$ is bi-radial if and only if $f$ is radial in both the $z$ and $t$ variables. For $k = 0, 1, 2, \cdots$ and $\lambda > 0,$ define the functions $\Phi_k^\lambda(z, t)$ by
$$ \Phi_k^\lambda(z, t) = C(k, n, m)~\varphi_k^\lambda(z)~\frac{J_{\frac{m}{2}-1}(s|t|)}{(s|t|)^{\frac{m}{2}-1}},$$ where $C(k, n, m)$ is a constant so that $\Phi_k^\lambda(0, 0) = 1.$ We have the following result about the class of integrable bi-radial functions, denoted by $L^1(G)^\# .$

\begin{thm}\label{biradial}
\begin{enumerate}
    \item The space $L^1(G)^\#$ is a commutative Banach algebra under convolution.
    \item The space of multiplicative linear functionals on $L^1(G)^\#$ coincides with the collection $\{\Phi_k^\lambda:~\lambda >0,~k =0, 1, 2, \cdots \}.$
    
\end{enumerate}

\end{thm}

For the proof of above see \cite{MR1703838} Proposition 5.3. We also need the product formula satisfied by the functions $\Phi_k^\lambda.$

\begin{Prop}\label{product}
Let $\Phi = \Phi_k^\lambda$ for some $k$ and $\lambda.$ Let $_{(z,t)}\Phi$ denote the left translate of the function $\Phi$ by the point $(z, t).$ Then, $$ \Pi(_{(z, t)}\Phi)(w, s) = \Phi(z, t)~\Phi(w, s).$$
\end{Prop}

For a proof, see Proposition 2.3 in \cite{MR1164603}. Now, a simple computation shows that the identity
in Proposition \ref{product} reduces to 

\[ \Phi_k^\lambda\ast \mu_{r,s}(z,t)=\Phi_k^\lambda(r,s)~\Phi_k^\lambda(z,t).\]
Choosing $\lambda$ and $ k > 0$ such that $\Phi_k^\lambda(r,s)=0,$ we get
$$ \Phi_k^\lambda \ast \mu_{r, s}(z, t) = 0~\forall~(z, t) \in G, $$
which proves our claim as $\Phi_k^\lambda(z,t)\in L^p$ if and only if $p>\frac{2m}{m-1}$.

Finally, we look at the  homogeneous spherical means defined using the measure $\sigma_r$. First we deal with the case $m \geq 2.$ Recall the homogeneous norm on $G,$ given by \[|(z,t)|=(|z|^4+|t|^2)^\frac{1}{4}. \]
Also, recall that there exists a unique Radon measure $\sigma$ on the unit sphere $\Sigma=\{(z,t):|(z,t)|=1 \}$ such that for all $f\in L^1(G)$
\[ \int\limits_G f(g)\, dg=\int\limits_0^\infty\int\limits_\Sigma f(\delta_r(z,t))\, d\sigma(z,t)\, r^{Q-1}dr\] where $\delta_r$ denote the dilations that act as automorphisms of $G.$ The measures $\sigma_r,$ for $ r> 0$ are defined by
\[ \sigma_r(f)=\sigma(\delta_r f)=\int\limits_\Sigma f(\delta_r(z,t))\, d\sigma(z,t).\]
The homogeneous spherical means of a function $f$ is then defined as the convolution $f \ast \sigma_r,$ of $f$ with $\sigma_r.$

 We have the formula for the measure $\sigma_s, s > 0$ given by \[\begin{split}
    \sigma_s(f) &=\int f(z,t)\, d\sigma_s(z,t) \\
    & = 2\int\limits_0^1\int\limits_{|z|=1}\int\limits_{|t|=1}~f(srz, s^2\sqrt{1-r^4}t)\, d\sigma(z)d\nu(t)\, r^{2n-1}~(1-r^4)^\frac{m-2}{2}~dr.
\end{split}\] where $\sigma$ and $\nu$ are the normalized surface measures on the unit spheres in $\mathbb{C}^n$ and $\mathbb{R}^m,$ respectively. See  \cite{fischer2008etude}*{Proposition 2.7} or \cite{MR2273787}*{p. 102} for the proof of this formula.
    
    It follows that
    \[f\ast \sigma_s=2\int\limits_0^1 ~f\ast \mu_{_{sr, s^2\sqrt{1-r^4}}}~ r^{2n-1}~(1-r^4)^\frac{m-2}{2}~dr  \]
    where $\mu_{sr, s^2\sqrt{1-r^4}}$ is the bi-spherical means defined earlier.
    
    As above we compute
    \[\begin{split}
       e_k^a \ast \sigma_s(z,t)& =2\int\limits_0^1 ~e_k^a\ast \mu_{sr, s^2\sqrt{1-r^4}}~(z,t)~r^{2n-1}~(1-r^4)^\frac{m-2}{2}~dr\\
        &= c_{k, n}~\left(2\int\limits_0^1 \frac{J_{\frac{m}{2}-1}(s^2\sqrt{1-r^4}|a|) }{(s^2\sqrt{1-r^4}|a|)^{\frac{m}{2}-1}}~\varphi_k^{|a|}(r)~r^{2n-1}~(1-r^4)^\frac{m-2}{2}~dr \right)\, e_k^a(z,t).
    \end{split} \] Write $|a|=\lambda$, and notice that the function
   \begin{equation}\label{eig}
   \lambda \mapsto \int\limits_0^1 ~\frac{J_{\frac{m}{2}-1}(s^2\sqrt{1-r^4}\lambda) }{(s^2\sqrt{1-r^4}\lambda)^{\frac{m}{2}-1}}~\varphi_k^{\lambda}(r)~r^{2n-1}~(1-r^4)^\frac{m-2}{2}dr
  \end{equation}
   is holomorphic for $\Re \lambda>0$ and so the above function has at most countably many zeros $\lambda\in (0,\infty)$. Now the proof can be completed as above for the range $1\leq p\leq \frac{2m}{m-1}$, if $m\geq 2$. We believe that the range obtained is optimal. This will be true if the function in \eqref{eig} has a zero in $(0, \infty).$

    When $m=1,$ $G=\mathbb{H}^n,$ the Heisenberg group. The formula for the measure $\sigma_s$ takes the following form (see \cite{MR898880}*{p. 95}:
    
$$ \sigma_s = c_n~\int_{-\frac{\pi}{2}}^{\frac{\pi}{2}}~\mu_{s \sqrt{\cos \theta}, \frac{1}{2}s^2 \sin \theta}~(\cos \theta)^{n-1}~d\theta,$$
where the measure $\mu_{r, s}$ is the normalized surface measure on the sphere $\{(z, s) \in \mathbb H^n:~|z| = r \}.$ Now the proof can be completed as earlier. We omit the details.
    This completes the proof of Theorem \ref{thirdinjectivity}.

\begin{remark}
The Abel summability result for the spectral decomposition will be true for all $1 < p < \infty,$ if we can estimate the operator norm of $\mathcal{A}_k.$
It is a natural question whether a two radius theorem is true for functions in $L^p(G)$ for $\frac{2m}{m-1} < p \leq \infty$ and whether our results can be proved for averages over $K$-orbits where $(G \rtimes K, K)$ is a Gelfand pair as in the case of the Heisenberg group. We hope to return to these questions and some others in the near future.

    \end{remark}

\begin{remark}
When $1 \leq p \leq 2,$ it is possible to take the Fourier transform in the central variable and prove the injectivity results for the spherical means with weaker conditions of growth on the function. See \cite{RR} 

\end{remark}

\bibliography{refs}
\end{document}